\documentclass[a4paper,12pt]{amsart}
\usepackage{amssymb}
\usepackage{amsmath, amsthm, amscd, amsfonts, amssymb, graphicx, color}
\usepackage{amsmath, amsthm, amscd, amsfonts, amssymb, graphicx, color}

\usepackage[cp1250]{inputenc}
\usepackage[T1]{fontenc}

\newtheorem{theorem}{Theorem}[section]
\newtheorem{lemma}[theorem]{Lemma}

\theoremstyle{definition}
\newtheorem{definition}[theorem]{Definition}
\newtheorem{example}[theorem]{Example}

\theoremstyle{remark}
\newtheorem{remark}[theorem]{Remark}
\numberwithin{equation}{section}

\begin{document}
	\author{Stefan Ivkovi\'{c} }
	
	\vspace{15pt}
	
	\title{Dynamics on weighted solid Banach function spaces}

	\maketitle
	\begin{abstract}
		The dynamics of weighted translation operators on Lebesgue spaces, Orlicz spaces, and in general on solid Banach function spaces have been studied in numerous papers. Recently, the dynamics of weighted translations on weighted Orlicz spaces have also been studied by Chen and others, extending thus the previous results concerning the dynamics of weighted translations on classical Orlicz spaces. The main idea of this paper is to obtain a generalization of these results to the case of general weighted solid Banach function spaces. More precisely, in this paper, we characterize disjoint topologically transitive weighted composition operators on weighted solid Banach function spaces. This approach has applications in the dynamics of weighted translations on weighted Morrey spaces.
	\end{abstract}
	
	\vspace{15pt}
	
	\begin{flushleft}
		\textbf{Keywords} Hypercyclic operator, supercyclic operator, weighted solid Banach function spaces, topological transitivity, weighted composition operator 
		
	\end{flushleft} 
	
	\vspace{15pt}
	
	\begin{flushleft}
		\textbf{Mathematics Subject Classification (2010)} Primary MSC 47A16, Secondary MSC 54H20.
	\end{flushleft}
	
	\vspace{30pt}
	
	\section{Introduction and preliminaries}

In this paper, the set of all Borel measurable complex-valued functions on a topological space $X$ is denoted by $\mathcal M_0(X) .$ Also, $\chi_A$ denotes the characteristic function of a set $A$. We recall the following definitions from \cite{solid}.
\begin{definition}
	Let $X$ be a topological space and $\mathcal{F}$ be a linear subspace of $\mathcal M_0(X)$. If $\mathcal F$ equipped with a given norm $\|\cdot\|_{\mathcal F}$ is a Banach space, we say that $\mathcal F$ is a \emph{Banach function space on $X$}.
\end{definition}
\begin{definition}
	Let $\mathcal{F}$ be a Banach function space on a topological space $X$, and  $\alpha:X\longrightarrow X $ be a homeomorphism. We say that $\mathcal{F}$ is \emph{$\alpha$-invariant} if for each 	$f\in \mathcal{F} $ we have 
	$f\circ \alpha^{\pm 1} \in \mathcal{F}$ and 	$\|f\circ \alpha^{\pm 1} \|_{\mathcal{F}}=\|f\|_{\mathcal{F}}$.
\end{definition}

\begin{definition} 
	Let $X$ be a topological space. Let $\alpha:X\longrightarrow X$ be invertible, and $\alpha,\alpha^{-1}$ be Borel measurable. We say that $\alpha$ is \emph{aperiodic} if for each compact subset $K$ of $X$, there exists a constant $N>0$ such that for each $n\geq N$, we have $K \cap \alpha^{n}(K)=\varnothing$, where $\alpha^{n}$ means the $n$-fold combination of $\alpha$. 
\end{definition}
Throughout this paper we will always assume that $ \alpha $ is an aperiodic homeomorphism of a topological space $X.$\\
Next, we recall also the following definition.
\begin{definition}
	A Banach function space $\mathcal{F}$ on $X$ is called \emph{solid} if for each $f\in \mathcal{F}$ and $g\in\mathcal M_0(X)$, satisfying $|g|\leq |f|$, we have $g \in \mathcal{F}$
	and $\|g\|_{\mathcal{F}}\leq \|f\|_{\mathcal{F}}$.
\end{definition}
For the next results we shall also assume that the following conditions from \cite{solid} on the Banach function space $\mathcal{F}$ hold.

\begin{definition}\label{condition}
	Let $X$ be a topological space, $\mathcal F$ be a Banach function space on $X$, and $\alpha$ be an aperiodic homeomorphism of $X .$ We say that $\mathcal F$ satisfies condition $\Omega_\alpha$ if the following conditions hold:
	~\	\begin{enumerate}
		\item $\mathcal{F}$ is solid and $\alpha$-invariant;
		\item for each compact set	$E\subseteq X$ we have $\chi_{E}\in \mathcal{F}$;
		\item $\mathcal{F}_{bc}$ is dense in $\mathcal F$, where $\mathcal{F}_{bc}$ is the set of all bounded compactly supported functions in $\mathcal{F}$.
	\end{enumerate}
\end{definition}

From now on, we shall assume that $\mathcal{F}$ is a Banach function space satisfying the conditions of all the definitions above. For a measurable positive function $ w $ on $X ,$ we let $ w^{-1} := \frac{1}{w}  .$ If $ w  $ is a positive measurable function on $  X $ such that $ w,w^{-1}  $ are bounded, then $ T_{\alpha,w} $ will denote the weighted composition operator on $ \mathcal{F}  $ defined by $ T_{\alpha,w} (f) = w \cdot (f \circ \alpha)  $ for all $ f \in  \mathcal{F}.  $  In this case,  $ T_{\alpha,w} $ is invertible, and we will let $ S_{\alpha,w} $  denote its inverse.  
\begin{enumerate}
	\item It is easy to verify that $T_{\alpha,w}$ and $S_{\alpha,w}$ are bounded linear operators on $\mathcal{F}$. From now on, we will simply denote $T:=T_{\alpha,w}$ and $S:=S_{\alpha,w}$.
	
	By calculation, we can confirm that for all $n \in \mathbb{N}$ and $f \in \mathcal{F}$, we have:
	
	$$T^{n}f = \prod_{j=0}^{n-1} (w \circ \alpha^{j}) . (f \circ \alpha^{n})$$
	and
	$$S^{n}f = \prod_{j=1}^{n} (w\circ \alpha^{-j})^{-1} \cdot (f \circ \alpha^{-n}).$$
	
	Hence, since $\|.\|_{\mathcal{F}}$ is $\alpha$-invariant, we obtain:
	$$\|T^{n}f\|_{\mathcal{F}} = \| (T^{n}f) \circ \alpha^{-n} \|_{\mathcal{F}} = \left\| \left( \prod_{j=1}^{n} (w \circ \alpha^{-j}) \right) . f \right\|_{\mathcal{F}}$$
	
	and 
	
	$$\|S^{n}f\|_{\mathcal{F}} = \| (S^{n}f) \circ \alpha^{n} \|_{\mathcal{F}} = \left\| \left( \prod_{j=0}^{n-1} (w \circ \alpha^{j})^{-1} \right) . f \right\|_{\mathcal{F}}.$$
\end{enumerate}

At the end of this section, we recall also the following definition.
\begin{definition} 
	Let $\mathcal X$ be a Banach space. A sequence $(T_n)_{n\in \mathbb{N}_0}$  of operators in $B(\mathcal X)$ is called {\it topologically transitive} if for each non-empty open subsets $U,V$ of
	${\mathcal X}$, $T_n(U)\cap V\neq \varnothing$ for some $n\in \mathbb{N}$. If $T_n(U)\cap V\neq \varnothing$ holds from some $n$ onwards, then
	$(T_n)_{n\in \mathbb{N}_0}$ is called {\it topologically mixing}.\\
	We say that $\{T_{n}\}_{n\in \mathbb{N}_0}$ is \textit{topologically transitive for positive supercyclicity} if for each non-empty open subsets $O_{1}$ and $O_{2}$ of $\mathcal{X}$, there exists some $n \in \mathbb{N}$ and some $\lambda \in \mathbb{R}^{+}$ such that $\lambda T_{n}(O_{1}) \cap O_{2} \neq \emptyset$\\
	A single operator $T$ in $B(\mathcal X)$ is called topologically transitive (respectively topologically mixing or topologically transitive for positive supercyclicity) if the sequence $(T^n)_{n\in \mathbb{N}_0}$ is topologically transitive (respectively topologically 
	mixing or topologically transitive for positive supercyclicity). \\
\end{definition}
\begin{remark} \label{remark-semi-transitive}
	
	If $T \in B(\mathcal{X}) $ is invertible  and topologically transitive for positive supercyclicity,  then, given two non-empty open subsets $ O_1 $ and $ O_2 $ of $ \mathcal{X} ,$ there exists a strictly  increasing sequence $\lbrace n_{k}\rbrace_{k} \subseteq \mathbb{N}$ and sequence $\lbrace \lambda_{k} \rbrace \subseteq \mathbb{R}^{+}$ such that for all $k \in \mathbb{N}$ we have that $$\lambda_{k} T^{n_{k}}(O_{1}) \cap O_{2} \neq \varnothing .$$ Indeed, since $T$ is topologically transitive for positive supercyclicity, we can find some $ n_{1} \in \mathbb{N}        $ and some  $  \lambda_{1} \in \mathbb{R}^{+}  $  such that   $\lambda_{1} T^{n_{1}}(O_{1}) \cap O_{2} \neq \varnothing        .$ Now, since   $  \lambda_{1} T^{n_{1}}      $   is invertible, it is an open map, hence   $  \lambda_{1} T^{n_{1}}(O_{1})      $  is open. Therefore, there exists some   $ \tilde{n}_{2} \in \mathbb{N}       $  and some   $\tilde{\lambda}_{2} \in  \mathbb{R}^{+} $  such that   $\tilde{\lambda}_{2} T^{{\tilde{n}}_{2}}(\lambda_{1} T^{n_{1}}(O_{1})) \cap O_{2} \neq \varnothing .$   Put   $n_{2}=n_{1}+\tilde{n}_{2} $  and   $  \lambda_{2} = \lambda_{1} \tilde{\lambda}_{2} .$ Proceeding inductively, we can construct the desired sequences   $ \lbrace n_{k} \rbrace_{k} $  and   $  \lbrace \lambda_{k} \rbrace_{k} .$ 
\end{remark}

\begin{definition}\label{defd--hyper}
	Let
	$N\geq 2$, and 	$T_{1},T_{2},\ldots,T_{N}$ 
	be bounded linear operators
	acting on a Banach space
	${\mathcal X}$. 
	\begin{enumerate}
		\item The finite sequence $T_{1},T_{2},\ldots,T_{N}$ is called \emph{disjoint hypercyclic} or simply \emph{d--hypercyclic} 
		if there exists an element $x\in \mathcal{X}$
		such that the set 
		\begin{equation}\label{15}
			\{(x,x,\ldots,x),(T_{1}x,T_{2}x,\ldots,T_{N}x),(T_{1}^{2}x,T_{2}^{2}x,\ldots,T_{N}^{2}x),\ldots\}
		\end{equation}
		is dense in $\mathcal{X}^{N}$. In this case, the element $x$ is called 
		a
		\emph{d--hypercyclic vector}. \\
		If the set of all d-hypercyclic vectors for $T_{1},T_{2},\ldots,T_{N}$ is dense in $X ,$ then we say that $T_{1},T_{2},\ldots,T_{N}$ are \emph{densely d-hypercyclic }.\\
		The finite sequence $T_{1},T_{2},\ldots,T_{N}$ is called \emph{disjoint supercyclic} or simply \emph{d--supercyclic} 
		if there exists an element $x\in \mathcal{X}$ and a sequence $ \lbrace \lambda_n \rbrace_n \subseteq \mathbb{C} $ 
		such that the set 
		\begin{equation}\label{15}
			\{(x,x,\ldots,x), \lambda_1(T_{1}x,T_{2}x,\ldots,T_{N}x),\lambda_2(T_{1}^{2}x,T_{2}^{2}x,\ldots,T_{N}^{2}x),\ldots\}
		\end{equation}
		is dense in $\mathcal{X}^{N}$. In this case, the element $x$ is called 
		a
		\emph{d--supercyclic vector}. \\
		If the set of all d-supercyclic vectors for $T_{1},T_{2},\ldots,T_{N}$ is dense in $X ,$ then we say that $T_{1},T_{2},\ldots,T_{N}$ are \emph{densely d-supercyclic }.
		\item The finite sequence $T_{1},T_{2},\ldots,T_{N}$ is called \emph{disjoint topologically transitive} or simply \emph{d--topologically transitive} 
		if for any non-empty open subsets $U,V_1,\ldots,V_N$ of $\mathcal X$, 
		there exist a natural number $n\in\mathbb N$ such that
		\begin{equation}\label{00}
			U\cap T_1^{-n}(V_1)\cap\cdots\cap T_N^{-n}(V_N)\neq\varnothing.
		\end{equation}
	\end{enumerate}
\end{definition}
We provide also the following auxiliary lemma.

\begin{lemma}\label{dis-trans}
	Let
	$N\geq 2$, 	$T_{1},T_{2},\ldots,T_{N}$ 
	be invertible bounded linear operators
	acting on a Banach space
	${\mathcal X}$ and $\{r_l\}_{l=1}^N$ be a finite sequence of natural numbers such that $0<r_1<r_2<\ldots<r_N .$ Suppose that $T_1^{r_{1}}, \dots, T_N^{r_{N}}$ are disjoint topologically transitive. Then, given any non-empty open subset $\mathcal{O}$ of $\mathcal{X}$, there exists a strictly increasing sequence $\{n_k\}_k \subseteq \mathbb{N}$ such that
	\[
	\mathcal{O} \cap \left( \bigcap_{\ell=1}^{N} T_\ell^{-r_{l} n_{k}}(\mathcal{O}) \right) \neq \emptyset \quad \text{for all } k \in \mathbb{N}.
	\]
\end{lemma}

\begin{proof}
	Since $T_1^{r_{1}}, \dots, T_N^{r_{N}}$ are disjoint topologically transitive, there exists some $n_1 \in \mathbb{N}$ such that
	\[
	\mathcal{O} \cap \left( \bigcap_{\ell=1}^{N} T_\ell^{-r_{l} n_{1}}(\mathcal{O}) \right) \neq \emptyset.
	\]
	Now, since $T_\ell^{-r_{l} n_{1}}$ is an open map for all $\ell \in \{1, \dots, N\}$, we have that $T_\ell^{-r_{l} n_{1}}(\mathcal{O})$ is open and non-empty for each $\ell \in \{1, \dots, N\}$, hence there exists some $\tilde{n}_2 \in \mathbb{N}$ such that 
	\[
	\mathcal{O} \cap \left( \bigcap_{\ell=1}^N T_\ell^{-r_{l} \tilde{n}_{2}} \left( T_\ell^{-r_{l} {n}_{2}}(\mathcal{O}) \right) \right) \neq \emptyset.
	\]
	Put $n_2 = \tilde{n}_2 + n_1$. Proceeding inductively, we deduce the statement of the lemma. 
	
\end{proof}

Next, we will give the concept of disjoint topological transitivity for positive supercyclicity, which is an extension of topological transitivity for positive supercyclicity. The following definition is in fact a modified version of \cite[Definition 1.2]{novo-prvi}.

\begin{definition} We say the operators  $ T_{1},T_{1}, \dots , T_{N}  \in \mathcal{B} (X)  $  are disjoint topologically transitive for positive supercyclicity if for every non-empty open subsets   $V_{0}, V_{1}, \dots V_{N}      $  of $X$, there exist   $ n \in \mathbb{N}     $   and   $\lambda \in \mathbb{R}^{+}      $  such that
	$$  V_{0} \cap (  \lambda T_{1}^{-n} ) (V_{1})  \cap \dots \cap (  \lambda T_{N}^{-n} ) (V_{N}) \neq \varnothing .   $$
\end{definition}
\begin{definition}
	Let $\mathcal{S}$ be a set, $\mathcal{L}$ be a family of subsets of $\mathcal{S} $   
	and $\{ T_{t,1} \}_{t \in \mathcal{S}}, \dots,
	\{ T_{t,N} \}_{t \in \mathcal{S}}$ be families of bounded linear operators on a Banach space $X$.  
	We say that $\{ T_{t,1} \}_{t \in \mathcal{S}}, \dots, \{ T_{t,N} \}_{t \in \mathcal{S}}$ are disjoint $\mathcal{L}$-semi-transitive, or shortly $d\mathcal{L}$-semi-transitive, if  
	for every collection of non-empty open subsets $\mathcal{O}, V_{1}, \dots, V_{N}$ of $X$,  
	there exists some $ F \in \mathcal{L} $  such that for all  $ t \in F$ there exists some $  \lambda_{t} \in \mathbb{R}^+  $  satisfying that 
	$$
	\mathcal{O} \cap \lambda_{t} T_{t,1}^{-1} (V_{1}) \cap \dots \cap \lambda_{t} T_{t,N}^{-1}(V_{N}) \neq \varnothing.
	$$
	
\end{definition}
We recall also that the families  $\{ T_{t,1} \}_{t \in S}, \dots, \{ T_{t,N} \}_{t \in S}$ are said to be \textit{disjoint $\mathcal{L}$-transitive, or shortly $d\mathcal{L}$-transitive}, if  
for every collection of non-empty open subsets $\mathcal{O}, V_{1}, \dots, V_{N}$ of $X$,  
there exists some $ F \in \mathcal{L} $  such that for all  $ t \in F$ we have
$$
\mathcal{O} \cap  T_{t,1}^{-1} (V_{1}) \cap \dots \cap  T_{t,N}^{-1}(V_{N}) \neq \varnothing.
$$

\section{Main results}
If $X$ is a topological space and $ \alpha $ is a homeomorphism on $ X ,$ then a weight on $X$ with respect to $\alpha$ is a continuous function $\eta : X \to (0, \infty)$ which satisfies
$$
\eta(\alpha(x)) \leq K_{\alpha} \eta(x) \quad \text{and} \quad \eta({\alpha}^{-1}(x)) \leq K_{\alpha} \eta(x) \quad \text{for all } x \in X,
$$
and some constant $K_{\alpha}
> 0$.\\
Below are some examples of such functions.

\begin{example}Let $X = \mathbb{R}$ and $p > 1$. Put
	$$
	\eta(x) = 
	\begin{cases}
		|x|^{p}\;\;\;  \text{for} & x \in  (1, \infty), \\
		1\;\;\;\;\;\;\; \text{for} & x \in [-1, 1], \\
		\frac{1}{|x|^{p}}\;\;\;  \text{for} & x \in (-\infty, -1).
	\end{cases}
	$$
	Then $\eta$ is a weight on $\mathbb{R}$ with respect to the translation $\alpha$ on $\mathbb{R}$ given by
	$$
	\alpha(x) = x + 1, \quad \text{for all } x \in \mathbb{R}.
	$$ Similar conclusion holds if we instead consider  $\alpha(x) = x - 1$, for all $x \in \mathbb{R}$.
\end{example}

\begin{example}\label{ilustracija}
	Let $X = \mathbb{R}^{n}$ for any natural number $n$ and $p > 1$. Put
	$$
	\eta(x) = 
	\begin{cases}
		
		1\;\;\;\;\;\;\; \text{for} & \|x \| \leq 1, \\
		\frac{1}{\|x \|^{p}}\;\;\;  \text{for} & \|x \| > 1.
	\end{cases}
	$$	If $ a \in \mathbb{R}^{n} $ with $ \|a \| =1 $ and $\alpha(x) = x - a$ for all $ x \in \mathbb{R}^{n} ,$ then $ \eta $ is a weight with respect to $ \alpha .$
\end{example}

The next example is motivated by \cite{banach}.

\begin{example}
	Let $G$ be a locally compact group and  $a \in G$ be a fixed element of $G$. Define $\alpha : G \to G$ by 
	$$
	\alpha(x) = ax \quad \text{for all } x \in G.
	$$
	If $\eta : G \to (0, \infty)$ is a continuous function which satisfies
	$$
	\eta(xy) \leq \eta(x)\eta(y) \quad \text{for all } x, y \in G,
	$$
	then $\eta$ is a weight on $G$ with respect to $\alpha.$
\end{example}

Let now $w : X \to (0, \infty)$ be a measurable function such that $w$ and $w^{-1}$ are bounded. We define the weighted solid Banach function space $\mathcal{F}_{\eta}$ on $X$ by
$$
\mathcal{F}_{\eta} = \left\{ f \in \mathcal{M}_{0}(X) \ \middle| \ f \eta \in \mathcal{F} \right\}
,$$
and we equip $\mathcal{F}_{\eta}$ with the norm $\|\cdot\|_{\mathcal{F}_{\eta}}$ given by
$$
\|f\|_{\mathcal{F}_{\eta}} = \|f \eta\|_{\mathcal{F}} \quad \text{for all } f \in \mathcal{F}_{\eta}.
$$

\begin{remark}\label{invariance} We notice that, while $\|\cdot\|_{\mathcal{F}}$ is $\alpha-$invariant, the norm $\|\cdot\|_{\mathcal{F}_{\eta}}$ does not need to be $\alpha-$invariant.
	
	\end{remark}

As in the previous section, we consider again the weighted composition operator $T_{\alpha, w}$ on $\mathcal{F}_{\eta}$ given by
$$
T_{\alpha, w}(f) = w \cdot (f \circ a) \quad \text{for all } f \in \mathcal{F}_{\eta}.
$$

\begin{lemma}
	The operator $T_{\alpha,w}$ is a bounded, linear self-mapping on $\mathcal{F}_{\eta}$.
\end{lemma}

\begin{proof}
	Put $M_w = \sup_{x \in X} w(x) .$ Then, since $\eta \circ \alpha \leq K_\alpha \eta$ and $\eta \circ \alpha^{-1} \leq K_\alpha \eta$,\\ we get
	$$
	T_{\alpha,w}(f) \cdot \eta =\eta \cdot w \cdot (f \circ \alpha)   = ((\eta \circ \alpha^{-1}) \cdot (w \circ \alpha^{-1}) \cdot f) \circ \alpha $$ $$
	\leq ((K_{\alpha} \eta)\cdot(w \circ \alpha^{-1})\cdot f) \circ \alpha \leq (M_w K_{\alpha} \eta \cdot f)  \circ \alpha
	$$
	for all $f \in \mathcal{F}_{\eta} .$ Since $\mathcal{F}$ is solid and $\alpha$-invariant, we deduce that
	$$
	T_{\alpha,w}(f) \cdot \eta \in \mathcal{F} \quad \text{and} \quad \| T_{\alpha,w}(f) \cdot \eta \|_{\mathcal{F}} \leq M_w K_{\alpha} \| f \cdot \eta \|_{\mathcal{F}} = M_w K_{\alpha} \|f\|_{\mathcal{F}_{\eta}},
	$$
	for all $f \in \mathcal{F}_{\eta} ,$ so $T_{\alpha}$ is a bounded linear operator on $\mathcal{F}_{\eta}$ and
	$$
	\| T_{\alpha,w} \| \leq M_w K_{\alpha}
	.$$
\end{proof}

Now we are ready to present the first theorem in this section.

\begin{theorem}\label{transitive} 
	Under the above notation and assumptions, the following statements are equivalent.
	\begin{enumerate}
		\item $T_{\alpha,w}$ is topologically transitive on $\mathcal{F}_{\eta}$.
		\item For each compact subset $K $ of $ X$, there exists a sequence of Borel subsets $\{E_{k}\}_{k}$ of $K$ and a strictly increasing sequence of natural numbers $\left\{ n_{k}\right\} _{k}$ such that
		$$
		\lim_{k \to \infty} \| \chi_{K\backslash E_{k}} \|_{\mathcal{F}} = 0
		$$
		and
		$$
		\lim_{k \to \infty} \sup_{x \in E_{k}} \left[ (\eta \circ \alpha^{n_k})(x) \cdot \left( \prod_{j=0}^{n_k - 1} (w\circ\alpha^j)^{-1}(x) \right) \right] =
		$$
		$$
		= \lim_{k \to \infty} \sup_{x \in E_{k}} \left[ (\eta \circ \alpha^{-n_k})(x) \cdot \left( \prod_{j=1}^{n_k} (w\circ\alpha^{-j})(x) \right) \right] = 0.
		$$
		If $ \alpha $ is not aperiodic, then we still have the implication $   (2)\rightarrow (1)   $ in that case.
	\end{enumerate}
\end{theorem}

\begin{proof}
	
	We prove first $(1) \Rightarrow (2)$. Let $K$ be a compact subset of $X$ and
	$$
	m_K = \inf_{x \in K} \eta(x).
	$$
	Since $\eta (X) \subseteq (0,\infty)$ and $\eta$ is continuous, we must have that $m_K > 0$. Now if $T_{\alpha,w}$ is topologically transitive, there exists some $n_1 \in \mathbb{N}$ such that
	$$
	T_{\alpha,w}^{n_1} \left( B(\chi_K, \tfrac{m_K}{4}) \right) \cap B(\chi_K, \tfrac{m_K}{4}) \neq \emptyset,
	$$
	(where $B(\chi_K, \tfrac{m_K}{4})$ is the open ball with centre in $\chi_K$ and radius $\tfrac{m_K}{4}$).
	
	By Lemma \ref{dis-trans} we can then find some $n_2 \in \mathbb{N}$ with $n_2 > n_1$ such that
	$$
	T_{\alpha,w}^{n_2} \left( B(\chi_K, \tfrac{m_K}{4^2}) \right) \cap B(\chi_K, \tfrac{m_K}{4^2}) \neq \emptyset.
	$$
	
	Proceeding inductively, we can find a strictly increasing sequence $\{n_k\} \subset \mathbb{N}$ such that
	$$
	T_{\alpha,w}^{n_k} \left( B(\chi_K, \tfrac{m_K}{4^k}) \right) \cap B(\chi_K, \tfrac{m_K}{4^k}) \neq \emptyset \quad \text{for each } k \in \mathbb{N}.
	$$
	
	Consequently, for each $k \in \mathbb{N}$ we have
	$$
	B(\chi_K, \tfrac{m_K}{4^k}) \cap S_{\alpha,w}^{n_k} \left( B(\chi_K, \tfrac{m_K}{4^k}) \right) \neq \emptyset.
	$$
	
	Therefore, we can find a sequence $\{f_k\}_k \subseteq \mathcal{F}_{\eta}$ such that for each $k \in \mathbb{N}$,
	$$
	\| f_k - \chi_K \|_{\mathcal{F}_\eta} < \frac{m_K}{4^k} \;\;\;\;\text{and} \quad \left\| S_{\alpha,w}^{n_k}(f_n) - \chi_K \right\|_{\mathcal{F}_\eta} < \frac{m_K}{4^k}.
	$$ As in the proof of \cite[Theorem 1]{solid} , for each $k \in \mathbb{N}$, put
	$$
	A_k = \left\{ x \in K : |f_k(x) - 1|  \geq \frac{1}{2^k} \right\}.
	$$
	
	Then, by solidity of $\mathcal{F}$, we get that for each $k \in \mathbb{N}$,
	$$	m_K \left| \chi_{A_k} (f_k - \chi_K ) \right|
	= m_K \left| \chi_{A_k} (\chi_K f_k - \chi_K) \right| 
	\leq m_K \left| \chi_K (f_k - \chi_K) \right| $$ 
	$$	\leq |\eta \chi_K (f_k - \chi_K)| 
	\leq |\eta (f_k - \chi_K)|, $$
	so, for each $k \in \mathbb{N}$ we obtain that $\chi_{A_k} f_k - \chi_{A_k} \in \mathcal{F}$
	and
	$$
	\left\| \chi_{A_k} f_k - \chi_{A_k} \right\|_{\mathcal{F}} \leq \frac{1}{m_K} \left\| \eta (f_k - \chi_K) \right\|_{\mathcal{F}} < \frac{1}{4^k}.
	$$
	
	By exactly the same arguments as in the proof of \cite[Theorem 1]{solid}, we deduce that
	$$
	\| \chi_{A_{k}}\| _{\mathcal{F}} \leq \tfrac{1}{2^k} \quad \text{for each } k \in \mathbb{N}.
	$$
	
	We also notice that since $\alpha$ is aperiodic, the sequence $\{n_k\}_k \subseteq \mathbb{N}$ can be chosen in such a way that $\alpha^{-n_k}(K) \cap K = \emptyset$ for all $k.$

	Put
	$$
	C_k = \left\{ x \in K : (\eta \circ \alpha^{n_k})(x) \cdot \left( \prod_{j=0}^{n_k - 1} (w \circ \alpha^j)(x) \right)^{-1} \cdot |f_k(x)| > \frac{m_K}{2^k} \right\}.
	$$
	
	Then $\chi_{C_k} (\alpha^{-n_k}(K)) = \lbrace 0 \rbrace $ for all $k$. Hence, we have that 
	$$
	\chi_{C_k} (\eta \circ \alpha^{n_k}) \prod_{j=0}^{n_k - 1} (w \circ \alpha^j)^{-1} \mid f_k \mid =
	$$
	$$= \Biggm|
	\chi_{C_k} (\eta \circ \alpha^{n_k}) \Biggm[ \prod_{j=1}^{n_k} (w \circ \alpha^{n_k - j})^{-1} f_k -( \chi_K \circ \alpha^{n_k}) 
	\Biggm] \Biggm|
	$$
	$$
	\leq
	\Biggm|  (\eta \circ \alpha^{n_k}) 
	\Biggm[ \prod_{j=1}^{n_k} (w \circ \alpha^{n_k - j})^{-1} f_k -(\chi_K \circ \alpha^{n_k}) \Biggm]  \Biggm|
	$$
	$$
	= (\eta \cdot \left| S_{\alpha,w}^{n_k}(f_k) - \chi_K \right|) \circ \alpha^{n_k}
	.$$
	
	Since $\mathcal{F}$ is solid and $\|\cdot\|_{\mathcal{F}}$ is $\alpha$-invariant, we deduce that
	\[
	\chi_{C_k} (\eta \circ \alpha^{n_k})  \prod_{j=0}^{n_k - 1} (w \circ \alpha^j)^{-1}  \left| f_k \right| \in \mathcal{F}
	\quad \text{and} \quad
	\left\| \chi_{C_k} (\eta \circ \alpha^{n_k})  \prod_{j=0}^{n_k - 1} (w \circ \alpha^j)^{-1}  f_k \right\|_{\mathcal{F}} 
	\]
	\[
	\leq \| \eta \cdot (S_{\alpha,w}^{n_k} (f_k) - \chi_K ) \|_{\mathcal{F}} \leq \frac{m_K}{4^k}.
	\]
	However, again by the solidity of $\mathcal{F}$ and the definition of $C_k$,
	\[
	\frac{m_K}{2^k} \| \chi_{C_k} \|_{\mathcal{F}} \leq \| \chi_{C_k} (\eta \circ \alpha^{n_k})  \prod_{j=0}^{n_k - 1} (w \circ \alpha^j)^{-1}  f_k \|_{\mathcal{F}} ,
	\]
	so we must have that $\| \chi_{C_k} \|_{\mathcal{F}} < \frac{1}{2^k}$ for all $k$.
	
	Let $x \in K \setminus (A_k \cup C_k)$. Then $|f_k(x)| < 1- \frac{1}{2^k}$, so
	\[
	(\eta \circ \alpha^{n_k})(x) \left( \prod_{j=0}^{n_k - 1} (w \circ \alpha^j)^{-1}(x) \right)
	\leq \frac{m_K}{2^{k}|f_k(x)|}  \leq \frac{m_K}{2^k-1} .
	\]
	
	Next, for each $k \in \mathbb{N}$, set
	\[
	g_k := S_{\alpha,w}^{n_k} f_k, \quad \text{then } \| g_k - \chi_K \|_{\mathcal{F}_{\eta}} < \frac{m_K}{4^k}
	\]
	and
	\[
	\| T_{\alpha,w}^{n_k} g_k - \chi_K \|_{\mathcal{F}_{\eta}} = \| f_k - \chi_K \|_{\mathcal{F}_{\eta}} < \frac{m_K}{4^k}
	\quad \text{for all } k \in \mathbb{N}.
	\]
	
	Put, for each $k\in\mathbb{N}$ 
	\[
	B_k := \left\{ x \in K : |g_k(x) - 1| \geq \tfrac{1}{2^k} \right\}.
	\]
	Then, by similar arguments as for $A_k$, we can show that
	\[
	\| \chi_{B_k} \|_{\mathcal{F}} \leq \tfrac{1}{2} \quad \text{for each } k \in \mathbb{N}.
	\]
	Further, for each $k \in \mathbb{N}$, set
	\[
	D_k := \left\{ x \in K : (\eta \circ \alpha^{-n_k})(x) \cdot \left( \prod_{j=1}^{n_k} (w \circ \alpha^{-j})(x) \right) \cdot |g_k(x)| > \tfrac{m_K}{2^k} \right\}.
	\]
	By similar arguments as for $C_k$ we can show that
	
	\[
	(\eta \circ \alpha^{-n_k}) \cdot \chi_{D_k} \cdot \left( \prod_{j=1}^{n_k} (w \circ \alpha^{-j}) \right) \cdot g_k \in \mathcal{F}
	\]
	and 
	\[
	\left|\left| (\eta \circ \alpha^{-n_k}) \cdot \chi_{D_k} \cdot \left( \prod_{j=1}^{n_k} (w \circ \alpha^{-j}) \right) \cdot g_k\right|\right|_\mathcal{F} \leq
	\]
	\[
	\leq \|\eta \cdot \left( T_{\alpha,w}^{n_k} g_k - \chi_K \right)\|_{\mathcal{F}} \leq \frac{m_K}{4^k}, \quad \text{for all } k.
	\]
	which further gives that
	\[
	\|\chi_{D_k}\|_{\mathcal{F}} \leq \frac{1}{2^k}, \quad \text{for all } k.
	\]
	
	If $x \in K \setminus (B_k \cup D_k)$, then $|g_k(x)| \geq 1 - \frac{1}{2^k}$, so
	\[
	(\eta \circ \alpha^{-n_k})(x) \cdot \left( \prod_{j=1}^{n_k} (w \circ \alpha^{-j})(x) \right)
	\leq \frac{m_K}{|g_k(x)| \cdot 2^k} \leq \frac{m_K}{2^k - 1}, \quad \text{for all } k \in \mathbb{N}.
	\]
	
	Set
	\[
	E_k := K \setminus (A_k \cup B_k \cup C_k \cup D_k).
	\]
	Then, by exactly the same arguments as in the proof of \cite[Theorem 1]{solid}, one can show that
	\[
	\|\chi_{K \setminus E_k}\|_{\mathcal{F}} \leq \frac{4}{2^k}, \quad \text{for all } k \in \mathbb{N}.
	\]
	Moreover, for each $x \in E_k$ and all $k \in \mathbb{N}$, we have
	\[
	(\eta \circ \alpha^{n_k})(x) \cdot \left( \prod_{j=0}^{n_k - 1} (w \circ \alpha^j)^{-1}(x) \right) \leq \frac{m_K}{2^k - 1}
	\]
	and
	\[
	(\eta \circ \alpha^{-n_k})(x) \cdot \left( \prod_{j=1}^{n_k} (w \circ \alpha^{-j})(x) \right) \leq \frac{m_K}{2^k - 1}.
	\]
	
	This proves $(1) \Rightarrow (2)$. Now we will prove the opposite implication.
	
	To this end, let $\mathcal{O}_1$ and $\mathcal{O}_2$ be two non-empty open subsets of $\mathcal{F}_\eta$. Choose some $f \in \mathcal{O}_1$ and $g \in \mathcal{O}_2$.
	Then there exist some $\varepsilon_1, \varepsilon_2 > 0$ such that the $\varepsilon_1$-neighbourhood of $f$ and $\varepsilon_2$-neighbourhood of $g$ are contained in $\mathcal{O}_1$ and $\mathcal{O}_2$, respectively. Since $f \eta , g \eta \in \mathcal{F}$ and $\mathcal{F}_{b_c} $ is dense in $\mathcal{F}$, we can find some $\tilde{f}, \tilde{g} \in \mathcal{F}_{ bc}$ such that
	\[
	\|\tilde{f} - f \eta\|_{\mathcal{F}} < \varepsilon_1, \quad \text{and} \quad \| \tilde{g} - g \eta\|_{\mathcal{F}} < \varepsilon_2.
	\]
	
	Set $K = \operatorname{supp} \tilde{f} \cup \operatorname{supp} \tilde{g}$, then $K$ is compact. Further,
	\[
	\tilde{f} \eta^{-1} , \tilde{g} \eta^{-1}   \in \mathcal{F}_\eta,
	\]
	in particular, $\tilde{f} \eta^{-1} \in \mathcal{O}_1$ and $\tilde{g}  \eta^{-1} \in \mathcal{O}_2$. In addition, $\operatorname{supp}\tilde{f} \eta^{-1} = \operatorname{supp} \tilde{f}$ and $\operatorname{supp}\tilde{g}\eta^{-1} = \operatorname{supp} \tilde{g}$. Put $\tilde{\tilde{f}} = \tilde{f} \eta^{-1} $ and $\tilde{\tilde{g}} = \tilde{g} \eta^{-1}$. Then $\tilde{\tilde{f}}  \chi_K = \tilde{\tilde{f}} $ and $\tilde{\tilde{g}} \chi_K = \tilde{\tilde{g}}$, $\tilde{\tilde{f}} \in \mathcal{O}_1$, and $\tilde{\tilde{g}} \in \mathcal{O}_2$.
	
	Choose the sequence $\{n_k\}_k \subseteq \mathbb{N}$ and the sequence $\{E_{k}\}_k$ of Borel subsets of $K$ that satisfy the assumptions of 2) with respect to $K$.
	
	Let again
	$$
	m_K := \inf_{x \in K} (\eta(x)) > 0.
	$$
	
	Then $m_K \tilde{\tilde{f}}\chi_{E_k} = m_K  \tilde{f}  \eta^{-1}  \chi_{E_k} \leq \eta  \tilde{f}  \eta^{-1}  \chi_{E_k}=  \tilde{f} \chi_{E_k}
	$ for all $k$ because $ E_k\subseteq K$ for each $ k ,$ so $m_K  \chi_{E_k} \leq \eta  \chi_{E_k} \quad \text{for all } k.$ Since $\mathcal{F}$ is solid and $\|\cdot\|_{\mathcal{F}}$ is $\alpha$-invariant, we deduce that
	
	\[
	\left\| T_{\alpha,w}^{n_k} \left( \tilde{\tilde{f}}  \chi_{E_k} \right) \right\|_{\mathcal{F}_\eta}
	= \left\| \eta  \left( \prod_{j=0}^{n_k - 1} (w \circ \alpha^j) \right)  \left( \tilde{\tilde{f}} \chi_{E_k} \circ \alpha^{n_k} \right) \right\|_{\mathcal{F}} 
	\]
	
	\[
	= \left\| (\eta \circ \alpha^{-n_k})  \left( \prod_{j=1}^{n_k} (w \circ \alpha^{-j}) \right)  \tilde{\tilde{f}} \chi_{E_k} \right\|_{\mathcal{F}}
	\]
	
	\[
	= \frac{1}{m_K} \left\| (\eta \circ \alpha^{-n_k}) \left( \prod_{j=1}^{n_k} (w \circ \alpha^{-j}) \right) m_K \tilde{\tilde{f}} \chi_{E_k} \right\|_{\mathcal{F}}
	\]
	\[
	\leq \frac{1}{m_K} \left\| (\eta \circ \alpha^{-n_k}) \left( \prod_{j=1}^{n_k} (w \circ \alpha^{-j}) \right) \tilde{f} \chi_{E_k} \right\|_{\mathcal{F}}
	\]
	\[
	\leq \frac{1}{m_K} \sup_{x \in E_k} \left( (\eta \circ \alpha^{-n_k})(x) \prod_{j=1}^{n_k} (w \circ \alpha^{-j})(x) \right) \| \widetilde{f} \|_{\mathcal{F}} \xrightarrow{k \to \infty} 0
	\]

	Similarly, we have that
	\[
	m_K \tilde{\tilde{g}}\chi_{E_k} \leq \widetilde{g} \chi_{E_k} \quad \text{for all } k,
	\]
	hence,
	\[
	\left\| S_{\alpha,w}^{n_k} (\widetilde{g} \chi_{E_k}) \right\|_{\mathcal{F}_{\eta}} = 
	\left\| \eta \left( \prod_{j=1}^{n_k} (w \circ \alpha^{-j})^{-1} \right) (\widetilde{g} \chi_{E_k} \circ \alpha^{-n_k}) \right\|_{\mathcal{F}} 
	\]
	\[
	= \frac{1}{m_K} \left\| (\eta \circ \alpha^{n_k}) \left( \prod_{j=0}^{n_k - 1} (w \circ \alpha^j)^{-1} \right) m_K \widetilde{g} \chi_{E_k} \right\|_{\mathcal{F}} 
	\leq \frac{1}{m_K} \left\| (\eta \circ \alpha^{n_k}) \left( \prod_{j=0}^{n_k - 1} (w \circ \alpha^j)^{-1} \right) \widetilde{g} \chi_{E_k} \right\|_{\mathcal{F}} 
	\]
	\[
	\leq \frac{1}{m_K} \sup_{x \in E_k} \left( (\eta \circ \alpha^{n_k})(x) \prod_{j=0}^{n_k - 1} (w \circ \alpha^j)^{-1}(x) \right) \| \widetilde{g} \|_{\mathcal{F}} \xrightarrow{k \to \infty} 0.
	\]

	Moreover,
	\[
	\left\| \tilde{\tilde{f}} \chi_{K \setminus E_k} \right\|_{\mathcal{F}_{\eta}} = \left\|  \widetilde{f} \chi_{K \setminus E_k} \right\|_{\mathcal{F}} 
	\leq \sup_{x \in X} (\widetilde{f}(x)) \left\| \chi_{K \setminus E_k} \right\|_{\mathcal{F}} \to 0 \quad \text{as } k \to \infty,
	\]
	and similarly,
	\[
	\left\| \tilde{\tilde{g}} \chi_{K \setminus E_k} \right\|_{\mathcal{F}_{\eta}} \to 0 \quad \text{as } k \to \infty.
	\]
	
	For each $k \in \mathbb{N}$, set
	\[
	v_k := \tilde{\tilde{f}} \chi_{E_k} + S_{\alpha,w}^{n_k} \left( \tilde{\tilde{g}} \chi_{E_k} \right).
	\]
	
	It follows from above that $v_k\to \tilde{\tilde{f}}$ and
	\[
	T_{\alpha,w}^{n_k}(v_k) \to \tilde{\tilde{g}} \text{ in } \mathcal{F}_{\eta}, \quad \text{as } k \to \infty.
	\]
	
	Since $\tilde{\tilde{f}} \in \mathcal{O}_1$ and $\tilde{\tilde{g}} \in \mathcal{O}_2$, we deduce that there exists some $k_0$ such that
	\[
	v_{k_0} \in \mathcal{O}_1 \quad \text{and} \quad T_{\alpha,w}^{n_{k_0}}(v_{k_0}) \in \mathcal{O}_2,
	\]
	which proves $(2) \Rightarrow (1)$.
\end{proof} 

The next example illustrates the difference between topological transitivity of weighted translations on solid Banach function spaces and topological transitivity of (weighted) translations on weighted solid Banach function spaces.

\begin{example}\label{weighted-primena}
	Let $ X = \mathbb{R} $ and $ \eta $ be the function on $ \mathbb{R} $ from Example \ref{ilustracija} with $ p=1.$ If $\alpha(x) = x - 1$, for all $x \in \mathbb{R}$ and $w$ is the constant function 1 on $\mathbb{R} ,$ then the operator $T_{\alpha, w}$ is topologically transitive on $\mathcal{F}_{\eta} $ by Theorem \ref{transitive}, however, by \cite[Theorem 2.12]{saw} it follows that $T_{\alpha, w}$ is not topologically transitive on $\mathcal{F} .$
\end{example}

\begin{remark}\label{weighted-remark}
	We notice that the special case of our results in this section are weighted translations on weighted Morrey spaces, see \cite[Section 4]{saw}. In fact, for any $ p,q$ with $ 1 \leq q < p < \infty ,$ we can let $ \mathcal{F} = \tilde{\mathcal{M}}_{q}^{p}( \mathbb{R}) $ in Example \ref{weighted-primena}. For the definition and details of the construction of the space $ \tilde{\mathcal{M}}_{q}^{p}( \mathbb{R}) ,$ see \cite[Section 4]{saw}.
\end{remark}

In what follows, given $ N \in \mathbb{N} ,$ we shall consider a finite family $\{\alpha_l\}_{l=1}^N $ of aperiodic homeomorphisms of $X$ which are in addition disjoint aperiodic, that is for each compact subset $K$ of $X$ there exists some $M \in \mathbb{N}$ such that $$ \alpha_s^{n}(K) \cap \alpha_\ell^{n}(K) = \emptyset $$ for all $ n \geq M$ and each distinct $s, \ell \in \{1, \ldots, N\} .$ We will assume that $ \eta $ is a weight with respect to $ \alpha_l $ for each $ \ell \in \{1, \ldots, N\} .$ Further, we let $\{w_l\}_{l=1}^N $ be a finite family of continuous positive functions on $X$ such that $w_l$ and $w_l^{-1}$ are bounded for each $ \ell \in \{1, \ldots, N\}$ and $\{r_l\}_{l=1}^N$ be a finite sequence of natural numbers such that $0<r_1<r_2<\ldots<r_N$.
To simplify notation, for each $ \ell \in \{1, \ldots, N\}$ we let $ T_\ell := T_{\alpha_\ell, w_\ell} .$  Under the above notation and assumptions, we present now the next theorem in this section, which is in fact an extension of \cite[Theorem 2.12]{saw} from the case of solid Banach function spaces to the case of weighted solid Banach function spaces.

\begin{theorem}\label{dis-weighted-transitive}
	The following statements are equivalent.
	\begin{enumerate}
		\item $T_1^{r_{1}},..., T_N^{r_{N}}$ are disjoint topologically transitive on $\mathcal{F}_\eta$.
		
		\item For each compact subset $K$ of $X$, there exists a strictly increasing sequence $\{n_k\}_k \subseteq \mathbb{N}$ and a sequence of Borel subsets $\{E_k\}_k$ of $K$ such that for all $ k \in \mathbb{N} $ and each distinct $ s, \ell \in \{1, \dots, N\} $ we have
		$$\alpha_l^{r_{l} n_{k}}(K) \cap K = \emptyset \quad \text{and} \quad
		\alpha_s^{r_{s} n_{k}}(K) \cap \alpha_\ell^{r_{l} n_{k}}(K) = \emptyset \quad \text{ }(4.1)$$ and such that $$\lim_{k \to \infty} \| \chi_{K \setminus E_k} \|_{\mathcal{F}} = 0. \text{ }(4.2)$$

		Moreover, for all $\ell \in \{1, \dots, N\}$, we have
		$$
		\lim_{k \to \infty} \sup_{x \in E_k} \left[ (\eta \circ \alpha_\ell^{{r_l n_k}})(x) \cdot
		\prod_{j=0}^{{r_l n_k}-1} (w_\ell \circ \alpha_\ell^j)^{-1}(x) \right]  =$$ 
		$$=\lim_{k \to \infty} \sup_{x \in E_k} \left[ (\eta \circ \alpha_\ell^{-{r_l n_k}})(x) \cdot 
		\prod_{j=1}^{{r_l n_k}} (w_\ell \circ \alpha_\ell^{-j})(x) \right] = 0, \text{ }(4.3) $$
		
		and for each distinct $s, \ell \in \{1, \dots, N\}$, it holds that
		$$
		\lim_{k \to \infty} \sup_{x \in E_k} \left[
		(\eta \circ \alpha_\ell^{-r_{l} n_{k}} \circ \alpha_s^{r_{s} n_{k}})(x) \cdot
		\frac{\prod_{j=1}^{r_{l} n_{k}} (w_\ell \circ \alpha_\ell^{-j} \circ \alpha_s^{r_{s} n_{k}})(x)}
		{\prod_{j=0}^{r_{s} n_{k}-1} (w_s \circ \alpha_s^j)(x)}
		\right] = 0. \text{ } (4.4)
		$$
	\end{enumerate}
\end{theorem}

\begin{proof}
	We prove first $(1) \Rightarrow (2)$. If $(1)$ holds, then by Lemma \ref{dis-trans} it is not hard to deduce that we can find a strictly increasing sequence $\{n_k\}_k \subseteq \mathbb{N}$ such that
	\[
	B(X_K, \tfrac{m_K}{4^k}) \cap \left( \bigcap_{\ell=1}^N S_\ell^{r_{l} n_{k}}(B(X_K, \tfrac{m_K}{4^k})) \right) \neq \emptyset
	\quad \text{for all } k \in \mathbb{N},
	\]
	(where again $m_K := \displaystyle\inf_{x \in K} \eta(x) > 0
	$, and $B(X_K, \tfrac{m_K}{4^k})$ is the open ball in $\mathcal{F}_\eta$ with centre in $X_K$ and radius $\tfrac{m_K}{4^k}$). In fact, since $\alpha_1, \dots, \alpha_N$ are disjoint aperiodic, we can construct the sequence $\{n_k\}_k \subseteq \mathbb{N}$ in a such way that (4.1) holds. Now, since, in particular,
	\[
	B\left(X_K, \frac{m_K}{4^k} \right) \cap S_\ell^{r_{l} n_{k}} \left( B\left(X_K, \frac{m_K}{4^k} \right) \right) \neq \emptyset
	\]
	for each $\ell \in \{1, \dots, N\}$ and all $k \in \mathbb{N}$, we can then proceed in a similar way as in the proof of $(1) \Rightarrow (2)$ in Theorem \ref{transitive} to obtain that for each  $ \ell \in \{1, \dots, N\} $ there exists a sequence of Borel subsets $\{R_{k,l}\}_k$ of $K$ such that    \[
	\|\chi_{K \setminus R_{k,l}}\|_{\mathcal{F}} \leq \frac{4}{2^k}, \quad \text{for all } k \in \mathbb{N}.
	\]
	and, moreover, for each $x \in R_{k,l}$ with $k \in \mathbb{N}$ and $ \ell \in \{1, \dots, N\}$, we have
	\[
	(\eta \circ \alpha_\ell^{r_l n_k})(x) \cdot \left( \prod_{j=0}^{r_l n_k - 1} (w_\ell \circ \alpha_\ell^j)^{-1}(x) \right) \leq \frac{m_K}{2^k - 1},
	\]
	
	\[
	(\eta \circ \alpha_\ell^{-r_l n_k})(x) \cdot \left( \prod_{j=1}^{r_l n_k} (w_\ell \circ \alpha_\ell^{-j})(x) \right) \leq \frac{m_K}{2^k - 1}.
	\].
	
	To simplify notation, for each $k \in \mathbb{N}$, we will denote 
	\[
	B_k := B\left( X_K, \frac{m_K}{4^k} \right).
	\]
	Since for each distinct $s, \ell \in \{1, \ldots, N\}$ and each $k \in \mathbb{N}$ we have that
	\[
	S_\ell^{r_{l} n_{k}}(B_k) \cap S_s^{r_{s} n_{k}}(B_k) \neq \emptyset ,
	\]

	for each distinct $s, \ell \in \{1, \ldots, N\}$ there exists a sequence $\{h_{\ell, s, k}\}_k \subseteq \mathcal{F}_\eta$ such that 
	\[
	\| h_{\ell,s,k} - \chi_K \|_{\mathcal{F}_\eta} < \frac{m_K}{4^k},
	\]
	and
	\[
	\left\| T_\ell^{r_{l} n_{k}} \left( S_s^{r_{s} n_{k}}(h_{\ell,s,k}) \right) - \chi_K \right\|_{\mathcal{F}_\eta} < \frac{m_K}{4^k} \quad \text{for all } k \in \mathbb{N}.
	\]
	
	For each distinct $s, \ell \in \{1, \dots, N\}$ and $k \in \mathbb{N},$ set
	
	\[
	F_{k,s,\ell} := \left\{ x \in K : 
	(\eta \circ \alpha^{-r_{l} n_{k}}_\ell \circ \alpha^{r_{s} n_{k}}_s)(x) 
	\cdot
	\frac{
		\displaystyle\prod_{j=1}^{r_{l} n_{k}} (w_\ell \circ \alpha^{-j}_\ell \circ \alpha_s^{r_{s} n_{k}})(x)
	}{
		\displaystyle\prod_{j=0}^{r_{s} n_{k}-1} (w_s \circ \alpha^j_s)(x)
	}
	\cdot
	\left| h_{\ell, s, k}(x) \right| \geq \frac{m_K}{2^k} 
	\right\} .
	\]
	
	Then, since 
	\[
	\alpha^{r_{l} n_{k}}_\ell(K) \cap \alpha^{r_{s} n_{k}}_s(K) = \emptyset \quad \text{for all } k,
	\]
	we get
	\[
	\chi_{F_{k,s,\ell}} \cdot \left( \chi_K \circ \alpha^{-r_{l} n_{k}}_\ell \circ \alpha^{r_{s} n_{k}}_s \right) = 0,
	\quad \text{hence}
	\]
	
	\[
	\left| \chi_{F_{k,s,\ell}} \cdot 
	(\eta \circ \alpha^{-r_{l} n_{k}}_\ell \circ \alpha^{r_{s} n_{k}}_s)
	\cdot
	\frac{
		\displaystyle\prod_{j=1}^{r_{l} n_{k}} (w_\ell \circ \alpha^{-j}_\ell \circ \alpha_s^{r_{s} n_{k}})
	}{
		\displaystyle\prod_{j=0}^{r_{s} n_{k}-1} (w_s \circ \alpha^j_s)
	}
	\cdot
	h_{\ell,s,k} \right|
	= 
	\]

	\[
	= \left| \chi_{F_{k,s,\ell}} \cdot 
	\left( \eta \circ \alpha_\ell^{-r_{l} n_{k}} \circ \alpha_s^{r_{s} n_{k}} \right)
	\cdot
	\left[ \frac{
		\displaystyle\prod_{j=1}^{r_{l} n_{k}} (w_\ell \circ \alpha_\ell^{-j} \circ \alpha_s^{r_{s} n_{k}})
	}{
		\displaystyle\prod_{j=0}^{r_{s} n_{k}-1} (w_s \circ \alpha_s^j)
	}
	\cdot
	h_{\ell,s,k} - \chi_K \circ \alpha_\ell^{-r_{l} n_{k}} \circ \alpha_s^{r_{s} n_{k}} \right]
	\right|
	\]
	
	\[
	\leq \left| 
	\left( \eta \circ \alpha_\ell^{-r_{l} n_{k}} \circ \alpha_s^{r_{s} n_{k}} \right)
	\cdot
	\left[ \frac{
		\displaystyle\prod_{j=1}^{r_{l} n_{k}} (w_\ell \circ \alpha_\ell^{-j} \circ \alpha_s^{r_{s} n_{k}})
	}{
		\displaystyle\prod_{j=0}^{r_{s} n_{k}-1} (w_s \circ \alpha_s^j)
	}
	\cdot
	h_{\ell,s,k} - \chi_K \circ \alpha_\ell^{-r_{l} n_{k}} \circ \alpha_s^{r_{s} n_{k}} \right]
	\right|
	\]
	
	\[
	= \left| \eta \cdot [T_\ell^{r_{l} n_{k}} \right( S_s^{r_{s} n_{k}}(h_{\ell,s,k}) \left) - \chi_K ]\right| \circ \alpha_\ell^{-r_{l} n_{k}} \circ \alpha_s^{r_{s} n_{k}} .
	\]
	
	Again, since $\mathcal{F}$ is solid and $\|\cdot\|_{\mathcal{F}}$ is $\alpha_j$-invariant $\text{for each } j \in \{1, \dots, N\}$, we deduce that 
	\[
	\chi_{F_{k,s,\ell}} \cdot 
	\left( \eta \circ \alpha_\ell^{-r_{l} n_{k}} \circ \alpha_s^{r_{s} n_{k}} \right)
	\cdot 
	\frac{
		\displaystyle\prod_{j=1}^{r_{l} n_{k}} \left( w_\ell \circ \alpha_\ell^{-j} \circ \alpha_s^{r_{s} n_{k}} \right)
	}{
		\displaystyle\prod_{j=0}^{r_{s} n_{k}-1} \left( w_s \circ \alpha_s^j \right)
	}
	\cdot h_{\ell, s, k} \in \mathcal{F}
	\]
	and 
	
	\[
	\frac{m_K}{2^k} \left\| \chi_{F_{k,s,\ell}} \right\|_{\mathcal{F}} \leq
	\]

	\[ \leq
	\left\| 
	\chi_{F_{k,s,\ell}} \cdot 
	(\eta \circ \alpha_\ell^{-r_{l} n_{k}} \circ \alpha_s^{r_{s} n_{k}}) \cdot
	\frac{
		\displaystyle\prod_{j=1}^{r_{l} n_{k}} (w_\ell \circ \alpha_\ell^{-j} \circ \alpha_s^{r_{s} n_{k}})
	}{
		\displaystyle\prod_{j=0}^{r_{s} n_{k}-1} (w_s \circ \alpha_s^j)
	} \cdot
	h_{\ell,s,k}
	\right\|_{\mathcal{F}}
	\]
	
	\[
	\leq  
	\left\|\eta \cdot \left[ 
	T_\ell^{r_{l} n_{k}} \left( S_s^{r_{s} n_{k}}(h_{\ell, s, k}) \right) - \chi_K 
	\right]\right\|_{\mathcal{F}} 
	\leq \frac{m_K}{4^k},
	\]
	
	\[ \text{so }
	\| \chi_{F_{k,s,\ell}} \|_{\mathcal{F}} \leq \frac{1}{2^k} \quad \text{for all } k \in \mathbb{N} \text{ and each distinct } s, \ell \in \{1, \dots, N\}.
	\]
	
	Further, for each \( k \in \mathbb{N} \) and each distinct \( s, \ell \in \{1, \dots, N\} \), set
	\[
	G_{k,s,\ell} = \left\{ x \in K : | h_{\ell,s,k}(x) - 1 | \geq \frac{1}{2^k} \right\}.
	\]
	
	As in the proof of $(1) \Rightarrow (2)$ in Theorem \ref{transitive}, since
	\[
	\| h_{\ell,s,k} - \chi_K \|_{\mathcal{F}_\eta} < \frac{m_K}{4^k},
	\]
	we deduce that
	\[
	\| \chi_{G_{k,s,\ell}} \|_{\mathcal{F}} < \frac{1}{2^k}.
	\]
	
	Moreover, for each \( x \in K \setminus (F_{k,s,\ell} \cup G_{k,s,\ell}) \), we have that $$ |h_{\ell,s,k}(x)| \geq 1 - \frac{1}{2^k}, \text{hence }  $$

	\[
	\left( \eta \circ \alpha^{-r_{l} n_{k}}_\ell \circ \alpha^{r_{s} n_{k}}_s \right)(x)
	\cdot 
	\frac{
		\displaystyle\prod_{j=1}^{r_{l} n_{k}} 
		(w_\ell \circ \alpha^{-j}_\ell \circ \alpha_s^{r_{s} n_{k}})(x)
	}{
		\displaystyle\prod_{j=0}^{r_{s} n_{k}-1}
		(w_s \circ \alpha_s^j)(x)
	}
	\leq
	\frac{m_K}{|h_{\ell,s,k}(x)| \cdot 2^k}
	\leq
	\frac{m_K}{\left(1 - \frac{1}{2^k} \right) \cdot 2^k}
	=
	\frac{m_K}{2^k - 1}
	\]
	
	\noindent
	\text{for all } $k \in \mathbb{N}$ \text{ and each distinct } $s, \ell \in \{1, \dots, N\}$. For each $ k \in \mathbb{N} ,$ set $ E_k := \left(K \setminus  \left( \bigcup_{\ell \neq s}^N (F_{k,s,\ell} \cup G_{k,s,\ell} ) \right) \right) \cup \left( \bigcup_{\ell=1}^N R_{k,l} \right) .$ Then, it is not hard to deduce that (2) holds.
	
	Next, we will prove	(2) $\Rightarrow$ (1). To this end, let $\mathcal{O}, \mathcal{O}_1, \dots, \mathcal{O}_N$ be open non-empty subsets of $\mathcal{F}_\eta$. As in the proof of 
	$(2) \Rightarrow (1)$ in Theorem \ref{transitive}, we can find some $\widetilde{f}, \widetilde{g}_1, \dots, \widetilde{g}_N \in \mathcal{F}_{bc}$ such that 
	$ \tilde{\tilde{f}} := \widetilde{f} \cdot \eta^{-1} \in \mathcal{O}$ and $\widetilde{\widetilde{g}}_\ell := \widetilde{g}_\ell \cdot  \eta^{-1} \in \mathcal{O}_\ell$ 
	for each $\ell \in \{1, \dots, N\}$. 
	
	Let $\{n_k\}_k \subseteq \mathbb{N}$ and $\{E_k\}_k$ be a sequence of Borel subsets of $$K := \mathrm{supp} \, \widetilde{f} 
	\cup \left( \bigcup_{\ell=1}^N \mathrm{supp} \, \widetilde{g}_\ell \right)$$ satisfying the assumptions of (2) with respect to $K$.
	By applying (4.2) and (4.3), it is not hard to deduce by the same arguments as in the proof of 
	$(2) \Rightarrow (1)$ in Theorem \ref{transitive} that 
	$\tilde{\tilde{f}} \chi_{E_k} \to\tilde{\tilde{f}} $ and 
	$\tilde{\tilde{g}}_\ell \chi_{E_k} \to \tilde{\tilde{g}}_\ell$ 
	in $\mathcal{F}_\eta$ as $k \to \infty$ for all $\ell \in \{1, \dots, N\}$. 
	Moreover, $T_\ell^{r_{l} n_{k}} \left( \tilde{\tilde{f}}  \chi_{E_k} \right) \to 0$ 
	and $S_\ell^{r_{s} n_{k}} \left( \tilde{\tilde{g}}_\ell \chi_{E_k} \right) \to 0$ 
	in $\mathcal{F}_\eta$ as $k \to \infty$, for each $\ell \in \{1, \dots, N\}$.\\
	Further, by (4.4), for each distinct $s, \ell \in \{1, \dots, N\}$ and all $ k \in \mathbb{N} ,$ we get
	
	\[
	\left\| T_\ell^{r_{l} n_{k}} \left( S_s^{r_{s} n_{k}} \left( \widetilde{\widetilde{g}}_s \chi_{E_k} \right) \right) \right\|_{\mathcal{F}_\eta}
	=
	\left\| \eta \cdot\left[ \prod_{j=0}^{r_{l} n_{k}-1} w_\ell \circ \alpha_\ell^j 
	\left( \prod_{j=1}^{r_{s} n_{k}} (w_s \circ \alpha_s^j)^{-1} (\widetilde{\widetilde{g}}_s \chi_{E_k}) \circ \alpha_s^{-r_{s} n_{k}} \right) \circ \alpha_\ell^{r_{l} n_{k}} \right] \right\|_{\mathcal{F}} 
	=
	\]
	\[
	= \left\| \left( \eta \circ \alpha_\ell^{-r_{l} n_{k}} \circ \alpha_s^{r_{s} n_{k}} \right) 
	\cdot \frac{ \displaystyle\prod_{j=1}^{r_{l} n_{k}} \left( w_\ell \circ \alpha_\ell^{-j} \circ \alpha_s^{r_{s} n_{k}} \right) }
	{ \displaystyle\prod_{j=0}^{r_{s} n_{k}-1} \left( w_s \circ \alpha_s^j \right) }
	\cdot \widetilde{\widetilde{g}}_s \chi_{E_k} \right\|_{\mathcal{F}}
	\leq 
	\]

	\begin{align*}
		\leq	\frac{1}{m_K} \left\| 
		\left( \eta \circ \alpha_\ell^{-r_{l} n_{k}} \circ \alpha_s^{r_{s} n_{k}} \right) \cdot 
		\frac{
			\displaystyle\prod_{j=1}^{r_{l} n_{k}} (w_\ell \circ \alpha_\ell^{-j} \circ \alpha_s^{r_{s} n_{k}})
		}{
			\displaystyle\prod_{j=0}^{r_{s} n_{k}-1} (w_s \circ \alpha_s^j)
		}
		\cdot \tilde{g}_s  \chi_{E_k} \right\|_{\mathcal{F}} 
		&\leq \\
		\frac{1}{m_K} \sup_{x \in E_k} \left( 
		\left( \eta \circ \alpha_\ell^{-r_{l} n_{k}} \circ \alpha_s^{r_{s} n_{k}} \right)(x) \cdot
		\frac{
			\displaystyle\prod_{j=1}^{r_{l} n_{k}} (w_\ell \circ \alpha_\ell^{-j} \circ \alpha_s^{r_{s} n_{k}})(x)
		}{
			\displaystyle\prod_{j=0}^{r_{s} n_{k}-1} (w_s \circ \alpha_s^j)(x)
		}
		\right)
		\left\| \tilde{g}_s \right\|_\infty \xrightarrow[k \to \infty]{} 0 .
	\end{align*}
	
	For each $k \in \mathbb{N}$, set 
	$$v_k := \tilde{\tilde{f}}  \chi_{E_k} + \sum_{s = 1}^{N} S_s^{r_{s} n_{k}} \left( (\widetilde{\widetilde{g}}_s  \chi_{E_k} \right).$$
	
	It follows that $v_k \to \tilde{\tilde{f}}$ and $T_\ell^{r_{l} n_{k}}(v_k) \to \widetilde{\widetilde{g}}_\ell$ in $\mathcal{F}_\eta$ as $k \to \infty$ for every $\ell \in \{1, \dots, N\}$.	
	
\end{proof}

\bibliographystyle{amsplain}

\end{document}